\UseRawInputEncoding
\documentclass[12pt,draft,reqno]{amsart}

\usepackage[all]{xy}
\usepackage{amsthm,array,amssymb,amscd,amsfonts,latexsym}

\usepackage{enumerate}

\usepackage{paralist}
\usepackage{mathrsfs}
\usepackage{dsfont}
\usepackage{bbold}
\usepackage{mathbbol}
\usepackage[all]{xy}
\usepackage{enumerate}

{\catcode`\@=11
\gdef\n@te#1#2{\leavevmode\vadjust{%
 {\setbox\z@\hbox to\z@{\strut#1}%
  \setbox\z@\hbox{\raise\dp\strutbox\box\z@}\ht\z@=\z@\dp\z@=\z@%
  #2\box\z@}}}
\gdef\leftnote#1{\n@te{\hss#1\quad}{}}
\gdef\rightnote#1{\n@te{\quad\kern-\leftskip#1\hss}{\moveright\hsize}}
\gdef\?{\FN@\qumark}
\gdef\qumark{\ifx\next"\DN@"##1"{\leftnote{\rm##1}}\else
 \DN@{\leftnote{\rm??}}\fi{\rm??}\next@}}

\DeclareFontFamily{OT1}{wncyr}{\hyphenchar\font45
}
\DeclareFontShape{OT1}{wncyr}{m}{n}{%
   <5> <6> <7> <8> <9> gen * wncyr
   <10> <10.95> <12> <14.4> <17.28> <20.74>  <24.88>wncyr10}{}
\DeclareFontShape{OT1}{wncyr}{m}{it}{%
   <5> <6> <7> <8> <9> gen * wncyi
   <10> <10.95> <12> <14.4> <17.28> <20.74> <24.88> wncyi10}{}
\DeclareFontShape{OT1}{wncyr}{m}{sc}{%
   <5> <6> <7> <8> <9> <10> <10.95> <12> <14.4>
   <17.28> <20.74> <24.88>wncysc10}{}
\DeclareFontShape{OT1}{wncyr}{b}{n}{%
   <5> <6> <7> <8> <9> gen * wncyb
   <10> <10.95> <12> <14.4> <17.28> <20.74> <24.88>wncyb10}{}
\input cyracc.def

\DeclareMathSizes{9}{9}{7}{5}

\theoremstyle{plain}

\newtheorem{theorem}{Theorem}[subsection]
\newtheorem{proposition}[theorem]{Proposition}

\newtheorem{lemma}[theorem]{Lemma}
\newtheorem{corollary}[theorem]{Corollary}

\theoremstyle{definition}

\newtheorem{definition}[theorem]{Definition}

\newtheorem{nothing*}[theorem]{}
\newtheorem{subnothing*}[sub]{}

\theoremstyle{remark}
\newtheorem{remark}[theorem]{\bf Remark}

\newcommand{\cc}{\raise .4pt \hbox{{$\scriptstyle{\bullet}$}}}

\begin{document}

\title[Faithful actions]{Faithful actions\\ of automorphism groups\\ of free groups\\
on algebraic varieties
}
\author[Vladimir L. Popov]{Vladimir L. Popov}
\address{Steklov Mathematical Institute,
Russian Academy of Sciences, Gub\-kina 8,
Moscow 119991, Russia}
\email{popovvl@mi-ras.ru}

\dedicatory{To the memory of J. E. Humphreys}


\begin{abstract}
\hskip -1mm
Considering a certain construction
of algebraic va\-rie\-ties $X$
en\-dowed with an algebraic action of the group
${\rm Aut}(F_n)$, $n<\infty$, we
obtain a criterion for the faithfulness of this action.
It gives an infinite family $\mathscr F$ of $X$'s such
that ${\rm Aut}(F_n)$ embeds into
${\rm Aut}(X)$. For $n\geqslant 3$, this implies nonlinearity, and for $n\geqslant 2$, the existence of $F_2$ in
${\rm Aut}(X)$ (hence nonamenability of the latter) for $X\in \mathscr F$.
We find in ${\mathscr F}$ two infinite subfamilies ${\mathscr N}$ and $\mathscr R$ con\-sist\-ing of irreducible affine varieties such
that every $X\in {\mathscr N}$ is
non\-rational (and even not stably rational),
while every $X\in \mathscr R$ is rational and $3n$-dimensional.
As an application, we show
that the minimal dimension of
affine algebraic varieties
$Z$, for which ${\rm Aut}(Z)$ contains the
braid group $B_n$ on $n\!\geqslant \!3$ strands,
does not exceed\;$3n$.
This upper bound strengthens the one
following from the paper by D. Krammer \cite{Kr02},
where the linearity of $B_n$ was proved (this latter bound is quadratic in $n$).
The same upper bound
also holds for ${\rm Aut}(F_n)$. In particular, it
shows that the minimal rank of the Cre\-mona groups containing
${\rm Aut}(F_n)$, does not exceed\;$3n$,
and the same is true for $B_n$ if $n\geqslant 3$.
 \end{abstract}

\maketitle

\renewcommand\thesubsection{\arabic{subsection}}
\subsection{Introduction}

 The exploration of abstract-algebraic, topological, algebro-geometric and dynamical properties of
 biregular auto\-morphism groups and
 birational self-map groups
 of algebraic va\-rie\-ties has become the trend of the last decade. In terms of popularity, the Cremona groups are probably the leaders among the studied groups.

Below, algebraic varieties and algebraic groups are understood in the same sense as in \cite{Se55}, \cite{Sh07}, \cite{Bo91},
\cite{Hu75} and are taken over  an
algebraically closed field $k$.

The subject of this paper are the following questions on the group embeddability related to automorphism groups of algebraic varieties.

\vskip 1.5mm

(Q1) For a given group $S$, is there an algebraic variety $Z$ such that $S$ embeds in the group ${\rm Aut}(Z)$ of its biregular automorphisms?

\vskip 1mm

(Q2) If yes, what are the properties of such $Z$? Are there such $Z$ in some distinguished classes of varieties (e.g., rational, nonrational, affine, complete, etc.)? What are the ``extreme'' values of the para\-me\-ters of such $Z$ (e.g., the minimum of their dimensions)? Etc.

\vskip 1mm

(Q3)  Conversely, in which groups can automorphism groups of algeb\-raic varieties of some type be embedded (e.g., are these groups linear?)

\vskip 1.5mm

Similar questions are also formulated in the context of birational self-map groups of algebraic varieties.

It is clear that question (Q1) (but not (Q2)) stands only for ``large'' groups $S$, in particular, nonlinear ones.
Generally speaking, the answer to it is no.\footnote{E.g., in view of
\cite[Thm.\,C]{CX18}, even in the context of birational self-map groups, the answer is negative if $S$ is an infinite simple torsion group with Kazhdan's property (T)
(such a group exists, see\,\cite[Sect.\;5]{Ki94}).}\label{foo}
Finding for a given $S$ the varieties $Z$ such that the answer is yes serves not only as a source of information  about
${\rm Aut}(Z)$, but also as the method of obtaining essential information about the structure of $S$ (see \cite{BL83}, \cite{Ma81}, \cite{CX18}).

In this paper, we explore the case $S={\rm Aut}(F_n)$,
where $F_n$ is a free group of rank $n<\infty$.
To this end, we consider a general construction
that assigns to any finitely generated group $\Sigma$
a family of algebraic varieties $Z$
endowed with an action of
${\rm Aut}(\Sigma)$ by biregular automor\-phisms.
Our results concern each of questions (Q1), (Q2), (Q3).
The main for us
is question (Q1), i.e., that
of faithfulness of
the action of ${\rm Aut}(F_n)$ on $Z$ which
means that the homomorphism
${\rm Aut}(F_n)\to {\rm Aut}(Z)$
defining the action is an embedding.

Here is the construction.
Let $\Sigma$ and $G$ be the groups, and  let
\begin{equation}\label{X}
X:={\rm Hom}(\Sigma, G).
\end{equation}
For any $\sigma\in {\rm End}(\Sigma)$,
$\gamma\in {\rm End}(G)$, put
\begin{equation}\label{definii}
  \sigma_X\colon X\to X,\;\; x\mapsto x\circ \sigma,\hskip 12mm
\gamma_X\colon X\to X,\;\; x\mapsto\gamma\circ x.
\end{equation}
If $\sigma\in {\rm Aut}(\Sigma)$ and $\gamma\in {\rm Aut}(G)$, then
$\sigma_X$ and $\gamma_X$ are invertible (their inverses $\sigma^{-1}_X$ and $\gamma_X^{-1}$ are respectively $
(\sigma^{-1})_X$ and
$(\gamma^{-1})_X$),\;and the mapping
 \begin{equation}\label{action}
 \big({\rm Aut}(\Sigma)\times {\rm Aut}(G)\big)\times X
 \to X,
 \quad (\sigma\gamma, x)\mapsto (\sigma_X^{-1}\circ\gamma_X)(x)
 \end{equation}
is an action on
$X$ of the group ${\rm Aut}(\Sigma)\times {\rm Aut}(G)$ (whose factors are naturally identified with its subgroups).

Below, when considering the action on $X$ of a subgroup of this group, the restriction of action \eqref{action} on it is always meant. The actions of
${\rm Aut}(\Sigma)$ and ${\rm Aut}(G)$ commute with each other.

If $\Sigma$ is a finitely generated group, and $G$ is an algebraic group, then $X$ is endowed with the structure of an algebraic variety so that all
$\sigma_X$ and $\gamma_X$ lie in ${\rm Aut}(X)$.
Let $R$ be an algebraic subgroup of ${\rm Aut}(G)$, for
whose action on $X$ there is a categorical quotient
\begin{equation}\label{Lunan}
\pi_{{X}/\!\!/R}^{\ }\colon {X}\to {X}/\!\!/R
\end{equation}
in the sense of geometric invariant theory (see \cite[Def. 05]{MF82}, \cite[Def.\,4.5]{PV94}).
The following two cases are the main  examples when
this quotient exists (see Proposition \ref{existe} below):
\begin{enumerate}[\hskip 4.2mm\rm (A)]
\item[\rm (F)] $R$ is finite;
\item[\rm (R)]  $G$ is affine and $R$ is reductive.
\end{enumerate}
Since the actions of ${\rm Aut}(\Sigma)$ and $R$ on $X$ commute,
it follows from the definition of categorical quotient that for every
$\sigma\in {\rm Aut}(\Sigma)$, the authomorphism
$\sigma_X$ of the variety $X$ descends to a uniquely defined automorphism
$\sigma_{X/\!\!/R}$ of the variety $X/\!\!/R$ having the property
\begin{equation}\label{pssp1}
\pi_{{X}/\!\!/R}^{\ }\circ\sigma_X=\sigma_{X/\!\!/R}^{\ }\circ\pi_{{X}/\!\!/R}^{\ }.
\end{equation}
The map
\begin{equation}\label{action//R}
{\rm Aut}(\Sigma)\to {\rm Aut}({X/\!\!/R}),\quad \sigma\mapsto\sigma_{X/\!\!/R}^{-1}
\end{equation}
is a group homomorphism.
It determines an action of ${\rm Aut}(\Sigma)$ on $X/\!\!/R$ by biregular automorphisms. In view of \eqref{pssp1}, the morphism $\pi_{{X}/\!\!/R}$ is
${\rm Aut}(\Sigma)$-equivariant.

In the present paper, for $\Sigma =F_n$, we consider the problem  of classify\-ing pairs $(G, R)$ such that the action of ${\rm Aut}(\Sigma)$ on $X/\!\! /R$ is
{\it faithful}. Our main results concern case (F).\footnote{In \cite{Po23}, we consider the situation of case (R) where $G$ is connected and se\-mi\-simple and $R$ is the image in ${\rm Int}(G)$ of a closed subgroup of a maximal torus of $G$. We prove the faithfulness of the action of ${\rm Aut}(F_n)$ on $X/\!\!/R$ in this case.} This problem is related to question (Q1). We apply our results
to questions (Q2), (Q3) as well.
These re\-sults consist of the following.

 The first is the faithfulness criterion for the action of ${\rm Aut}(F_n)$ on ${X}/\!\!/R$
  in case (F).

\begin{theorem}\label{fg1}
Let $G$ be an algebraic group
\textup(not necessarily connected or affine\textup), $X={\rm Hom}(F_n, G)$, $n\geqslant 2$, and let $R$ be a finite subgroup of
 ${\rm Aut}(G)$.
The following properties are equivalent:
\begin{enumerate}[\hskip 3.7mm \rm(a)]
\item
the action of ${\rm Aut}(F_n)$ on $X/\!\!/R$ is faithful;
\item the connected component of the identity of the group $G$ is non\-sol\-v\-able.
 \end{enumerate}
\end{theorem}

   Corollary \ref{cor0n}
    describes the applications of
    Theorem \ref{fg1}
    to questions (Q1)--(Q3), namely,
    to the problem of linearity of automorphism groups of algebraic varieties (considered in \cite[Prop. 5.1]{CD13}, \cite{Co17}, \cite{Ca12}), and to the problem of describing subgroups of the Cremona groups.

\begin{corollary}\label{cor0n} In the notation of Theorem
{\rm \ref{fg1}}, let the connected com\-po\-nent of the identity of the group $G$ be nonsolvable. Then
\begin{enumerate}[\hskip 4.2mm \rm(a)]
\item\label{aaaa}
${\rm Aut}(X/\!\!/R)$ contains the following groups:
\begin{enumerate}[\hskip 0mm $\cc$]
\item
${\rm Aut}(F_n)$,
\item
$F_2$,
 \item
 the braid group $B_n$ on $n$ strands if
 $n\geqslant 3$;
 \end{enumerate}
 \item\label{bbbb}
 ${\rm Aut}(X/\!\!/R)$ is nonamenable and, if $n\geqslant 3$, nonlinear.
 \end{enumerate}
\end{corollary}

Among the varieties $X/\!\!/R$
from Corollary \ref{cor0n}  whose automorphism group contains
${\rm Aut}(F_n)$, $F_2$ and $B_n$, there are both rational and nonra\-tional (and even not stably rational), namely:

\begin{proposition}\label{prpr}
Let, in the notation of Theorem {\rm \ref{fg1}}, the group $G$ be connected and the group $R$ be trivial. Then the variety $X/\!\!/R$ is rational if $G$ is affine and nonunirational,
if $G$ is nonaffine.
\end{proposition}

In general case, the rationality of the variety $X/\!\!/R$ for a connected reductive $G$ and $R={\rm Int}(G)$ is the old
problem open even for $n=2$ and
$G={\rm GL}_d$ with $d\geqslant 5$ (see \cite[(1.5.2)]{Po94}, \cite[pp.\;190--191]{DF04}).

In view of Proposition \ref{prpr}, if $G$ is nonaffine, then the variety $X/\!\!/R$ with trivial $R$ is not stably rational. For nontrivial finite $R$, the variety $X/\!\!/R$ with the faithful action of ${\rm Aut}(F_n)$ may be not stably rational even if $G$
is affine. Our second main result is Theorem \ref{stratio1} giving the construction of such affine $X/\!\!/R$ with a connected reductive $G$.
Its proof also uses Theorem\;\ref{fg1}.

\begin{theorem}\label{stratio1}
 For every prime number $p\neq {\rm char}(k)$, there is a finite $p$-group $K$,
having the following property. Let
$V$ be a finite-dimensional vector space over $k$, and let
$\iota\colon\! K\!\hookrightarrow\! {\rm GL}(V)$ be a group embedding for which
$\iota(K)$ has no nontrivial center elements of the group ${\rm GL}(V)$ \textup(such pairs $(V,\iota)$ exist for any finite group $K$\textup).
Let $X={\rm Hom}(F_n, {\rm GL}(V))$, $n\geqslant 2$, and let
$R$ be the image of the group $\iota(K)$ under the canonical homomorphism
${\rm GL}(V)\to {\rm Int}\big({\rm GL}(V)\big)$.
Then $X/\!\!/R$ is  nonrational \textup (and even not stably
 rational\textup) affine algebraic variety,
 on which the group ${\rm Aut}(F_n)$ acts faithfully.
\end{theorem}

Examples of groups $K$ from Theorem \ref{stratio1} can be explicitly specified using generators and relations (see Remark \ref{stra} below).

Our third main result concerns question (Q2) and, in particular, gives upper bounds for ``extremal'' parameter values for embeddings of ${\rm Aut}(F_n)$ and $B_n$ into automorphism groups of algebraic varieties.

\begin{theorem}\label{Cr1nn} Keep the notation of Theorem {\rm \ref{fg1}}.\;Let
  $n\geqslant 2$, $G={\rm SL}_2$ or ${\rm PSL}_2$, and let $R$ be finite. Then $X/\!\!/R$ is an irreducible rational affine $3n$-dimensional algebraic variety, whose automorphism group
contains ${\rm Aut}(F_n)$.
\end{theorem}

Note that the variety $X/\!\!/R$ from Theorem \ref{Cr1nn} in the case of trivial $R$ and $G={\rm SL}_2$ (respectively, ${\rm PSL}_2$)
is the product of $n$ copies of the smooth affine quadric $Q$ in $\mathbb A^4$ given by the equation $x_1x_2+x_3x_4=1$
(respectively, $n$ copies of $Q/I$, where $I$ is the group,
generated by the automorphism $(a,b,c,d)\mapsto (-a,-b,-c,-d)$).

\begin{definition}\label{Defff}
For any group $S$, denote by ${\rm Var}_k(S)$ (respectively,
${\rm Crem}_k(S)$) the minimal dimension of (defined over $k$) irreducible algeb\-raic varieties $Z$ (respectively, the ranks of the Cremona groups $C$) such that $S$ embeds into ${\rm Aut}(Z)$ (respectively, into $C$). If there are no such $Z$ (respectively, $C$), then set ${\rm Var}_k(S)=\infty$ (respectively, ${\rm Crem}_k(S)=\infty$).
\end{definition}

Groups $S$ with ${\rm Var}_k(S)\!={\rm Crem}_k(S)\!=\!\infty$ exist (see footnote\,${}^1$).

\begin{corollary}\label{boundsn} Let $S={\rm Aut}(F_n)$ with $n\geqslant 1$ or $B_n$ with $n\geqslant 3$. Then
\begin{equation}\label{esti}
\mbox{${\rm Var}_k(S)\leqslant 3n$\quad and\quad
${\rm Crem}_k(S)\leqslant 3n$.}
\end{equation}
\end{corollary}

The upper bounds \eqref{esti} for $S=B_n$
strengthen the ones
following from the paper by D. Krammer \cite{Kr02},
where embeddability of $B_n$ into ${\rm GL}_{n(n-1)/2}$
was proved (which yields the upper bound ${n(n-1)/2}$).\footnote{As the referee noted,
a lower bound for ${\rm Var}_k({\rm Aut}(F_n))
$ can be obtained using the methods of \cite{CX18}.}

A special case of the described construction,
where $R={\rm Int}(G)$ with reductive $G$, is explored
in many publications, starting essentially with the paper by Vogt of 1889. The subjects of these studies
 are: (a) ap\-pli\-cations to the theory of deformations of hyperbolic structures on to\-po\-logical surfaces, see \cite{Go09}
 (in this case, $\Sigma$ is the fundamental group of the surface, $G={\rm SL}_2(\mathbb C)$, and $X/\!\!/R$ is the ``variety of characters'' of $\Sigma$);
(b) dynamic properties of the action of ${\rm Aut}(\Sigma)$ on
$X/\!\!/R$, see \cite{Go97}, \cite{Go06}, \cite{Ca131};
(c)  for $\Sigma=F_n$, finding the equations of the ``variety of cha\-rac\-ters'' and describing the kernel of the action of ${\rm Aut}(F_n)$ on it, see \cite{Ho72}, \cite{Ho75}, \cite {Ma80}.

For the purposes of this paper, this special case is of little interest, since the group ${\rm Int}(\Sigma)$ is always contained in the kernel of the action of
${\rm Aut}(\Sigma)$ on $X/\!\!/{\rm Int}(G)$, and therefore,
for $\Sigma=F_n$, the faithfulness of this action is possible only for $n=1$ (when ${\rm Aut}(F_n)$ is a group of order $2$).
For $n=1$ and connected $G$, the rare cases when this action is faithful  are described the following theorem.

\begin{theorem}\label{S=Gn}
Let $G$ be a connected reductive algebraic group, $X\!=\!{\rm Hom}(F_n, G)$ and $R\!=\!{\rm Int}(G)$. The action of ${\rm Aut}(F_n)$ on $X/\!\!/R$ is faithful if and only if
$n=1$ and $G$ contains a connected simple normal subgroup
any of the following types:
\begin{equation}\label{type}
\mbox{${\sf A}_{\ell}$ with $\ell\geqslant 2$,\;
${\sf D}_\ell$ with odd $\ell$, \;
${\sf E}_6$.}
\end{equation}
\end{theorem}

The proof of Theorem \ref{fg1} is given in Sections \ref{p0} and \ref{p1},
of Corollary \ref{cor0n} in Section \ref{p2}, of Theorem \ref{stratio1} in Section  \ref{p3},
of Theorem \ref{Cr1nn} and Corollary \ref{boundsn} in Section \ref{p4}, and of
Theorem \ref{S=Gn} in Section \ref{p5}.

\vskip 2mm

{\it Acknowledgements.} The author is grateful to N.\;L.\;Gordeev for dis\-cussing the questions about group identities that arose in connection with the proof of Theorem \ref{fg1}, information about some publications on this topic, and comments on the first version \cite{Po21} of this paper.
The author is also grateful to the referee whose comments are highly appreciated.

\subsection{Conventions and notation}\

If $X$ is an algebraic variety (respectively, a differentiable manifold),
then ${\rm Aut}(X)$ denotes the group
its regular automorphisms
(respecti\-vely, diffemorphisms).

Groups are considered in multiplicative notation.
The identity ele\-ment of a group is denoted by $e$ (it is clear from the context which group is meant).

The claim that the group $G$ contains the group $H$,
means the exis\-tence of a group embedding $\iota\colon H\hookrightarrow G$, by which $H$ is identified with $\iota(H)$.

${\mathscr C}(G)$ is the center of the group  $G$.

${\mathscr C}_G(g)$ is the centralizer in $G$ of an element $g\in G$.

${\rm int}_g$ is the inner group automorphism determined by an element  $g$.

$\langle g_1,\ldots, g_m\rangle$ is the group generated by the elements  $g_1,\ldots, g_m$.

$G^0$ is the connected component of the identity of an algebraic group or a real Lie group $G$.

The Lie algebra of an algebraic group is denoted by the lowercase Gothic version of the letter denoting that group.

$\underline{G}$ is the underlying variety (or manifold) of an algebraic group (or real Lie group) $G$.

${\rm Aut}(G)$, ${\rm  Int}(G)$, ${\rm Out}(G)$, and ${\rm End}(G)$ are respectively the group of automorphisms, inner automorphisms, outer automorphisms, and the monoid of endomorphisms of a group $G$. If $G$ is an algebraic group or a real Lie group,
then by its automorphisms we mean automorphisms in the category of algebraic groups or real
Lie groups, so that ${\rm Aut}(G)$ denotes the intersection of ${\rm Aut}({\underline G})$ with the automorphism group of the abstract group $G$.
If an algebraic group $H$ faithfully acts by automor\-phisms of an algebraic group $G$ and the mapping $H\times G\to G$  defining this action
is a morphism of algebraic varieties, then
$H$ is called an algebraic subgroup of ${\rm Aut}(G)$.

The reductivity of an affine algebraic group $G$ does not assume its connectedness and is understood in the sense of \cite{MF82}, i.e., as the triviality of the unipotent radical of the group $G^0$.

\subsection{Fixing a system of generators of \boldmath $\Sigma$
}\label{wds}

Consider a group $G$ and a finitely generated group
$\Sigma$. Let $s_1,\ldots, s_n$ be a system of
generators of $\Sigma$ and let
$\varphi\colon F_n\to \Sigma$ be the epimorphism
defined by the equalities $\varphi(f_j)=s_j$ for every $j$.

For any group $H$
and any $w\in F_n$,  $h=(h_1,\ldots, h_n)\in H^n$, denote by $w(h)=w(h_1,\ldots, h_n)$ the image of $w$ under the (unique) homomor\-phism $F_n\to H$ mapping $f_j$ to $h_j$ for every $j$. In other words, if we write $w$ as a word
\begin{equation}\label{Lora}
f_{i_1}^{\varepsilon_1}\cdots f_{i_d}^{\varepsilon_d} ,\quad\mbox{where $\varepsilon_j\in\mathbb Z$},
\end{equation}
(a noncommutative Laurent monomial in $f_1,\ldots, f_n$), then $w(h)$ is ob\-tained by replacing $f_j$ with $h_j$ in
\eqref{Lora} for each $j$.

The map
\begin{equation}\label{embe}
X:={\rm Hom}(\Sigma, G)\to G^n,\quad x\mapsto \big(x(s_1),\ldots, x(s_n)\big)\in G^n
\end{equation}
is an injection. Its image is the set
\begin{equation}\label{image}
 \{g\in G^n\mid \mbox{$w(g)=e$ for all $w\in {\rm Ker}(\varphi)$}
 \}.
\end{equation}
In the rest of this paper, if necessary and without reminders, we {\it identify $X$ with the set} \eqref{image} using the injection \eqref{embe}.
For $\Sigma=F_n$ and $s_j=f_j$ for all $j$, we have
$X=G^n$, so in this case $X$
is the group (with the componentwise multiplication).

Let $g=(g_1,\ldots, g_n)
\in X\subseteq G^n$ and $t\in \Sigma$. It follows from
\eqref{image} that the element $w(g)\in G$ is the same for all $w\in \varphi^{-1}(t)$. Denote it by $t(g)$. In other words, writing $t$ as a noncommutative Laurent monomial in $s_1,\ldots, s_n$ and replacing $s_j$ in this monomial by $g_j$ for each $j$, we obtain, regardless of the chosen monomial, $t( g)$. In this notation, for any $\sigma\in {\rm End}(\Sigma)$, $\gamma\in {\rm End}(G)$,  formulas \eqref{definii} are rewritten as follows:
\begin{equation}
\begin{split}
  \sigma_X&\colon X\to X,\quad g=(g_1,\ldots, g_n)\mapsto \big(\sigma(s_1)(g),\ldots, \sigma(s_1)(g)\big),\\[-.5mm]
  \gamma_X&\colon X\to X, \quad (g_1,\ldots, g_n)\mapsto \big(\gamma(g_1),\ldots, \gamma(g_n)\big).
  \end{split}\label{mornew}
\end{equation}

If $G$ is an algebraic group, then $X$
is closed in $G^n$ and therefore
endowed with the structure of an algebraic variety. This structure does not depend on the choice of systems of generators, and the maps \eqref{mornew} are
morphisms. If $\Sigma=F_n$ and $G$ is a real Lie group, then
\eqref{mornew} are the differentiable mappings $G^n\to G^n$.

Some properties, selectively used below and in \cite{Po21}, \cite{Po23}, \cite{Po22}, are brought together in Proposition \ref{pro} for ease of reference.

\begin{proposition}\label{pro}
We maintain the notation introduced above in Sec\-tion {\rm \ref{wds}}. Let $\sigma$ and $\tau\in {\rm End}(F_n)$. Then the following hold.
\begin{enumerate}[\hskip 4.2mm\rm(a)]
\item\label{n1}
${(\sigma\circ \tau)_{X}^{\ }}=\tau_{X}\circ \sigma_{X}^{\ }$.
 \item\label{n2}
    $e_X={\rm id}$.
\item\label{n3}
$\sigma_{X}(X\cap S^{n})
\subseteq X\cap S^{n}$ for any subgroup $S$ of the group $G$.
\item\label{n6} Let $\theta\colon G\!\to\! H$ be a group homomorphism
and let $Y\!:=\!{\rm Hom}(\Sigma, H)\break \subseteq H^{n}$. Then the map
\begin{align*}
{\theta}_n\colon X
\to Y, \quad
(g_1,\ldots, g_n)
\mapsto (\theta(g_1),\ldots, \theta(g_n))
\end{align*}
is ${\rm End}(\Sigma)$-equivariant, i.e.,
${\theta}_n\circ \sigma_{X}=\sigma_{Y}\circ {\theta}_n.$
\item \label{n5}
If $\sigma={\rm int}_t$ for $t\in \Sigma$, then
the following properties of an element
\begin{equation}\label{gh}
x=(g_1,\ldots, g_n)\in X\subseteq G^n
\end{equation}
are equivalent:
\begin{enumerate}[\hskip 0mm\rm(a)]
\item[{\rm (\ref{n5}${}_1$)}] $\sigma_X(x)=x$;
\item[{\rm (\ref{n5}${}_1$)}] $t(x)\in \bigcap_{i=1}^n{\mathscr C}_G(g_i)$.
\end{enumerate}
\end{enumerate}
In statements {\rm (f)} and {\rm (g)}, it is assumed that $\Sigma=F_n$.
\begin{enumerate}[\hskip 4.2mm\rm(f)]
\item[\rm(f)]\label{n4}
The following properties of element  {\rm \eqref{gh}} are equivalent:
\begin{enumerate}[\hskip 0mm\rm(a)]
\item[{\rm (f${}_1$)}] $\sigma_X(x)=x$ for each $\sigma\in {\rm Aut}(F_n)$;
\item[{\rm (f${}_2)$}] if $n>1$, then $g_1=\cdots=g_n=e$, and if $n=1$, then $g_1^2=e$.
\end{enumerate}
\item[\rm(g)]\label{n7} The multiplication in $X=G^n$ has the property:
\begin{equation*}
\mbox{$\sigma_{X}(xz)=\sigma_{X}(x)\sigma_{X}(z)$ for all $x\in
X=G^n$, $z\in {\mathscr C}(X)$.}
\end{equation*}
 In particular, the restriction of
$\sigma_{X}$ to the group
${\mathscr C}(X)$ is its endo\-mor\-phism.
\end{enumerate}
\end{proposition}

\begin{proof}
Statement (f) follows from the fact that for $n=1$, the only nonidentity
element of ${\rm Aut}(F_n)$ maps $f_1$ to $f_1^{-1}$, and
for $n\geqslant 2$,
for any $i, j\in\{1,\ldots, n\}$, $i\neq j$,
the element $\sigma_{ij}\in{\rm End}(F_n)$ defined by the formula
\begin{equation*}
\sigma_{ij}(f_{l})=\begin{cases}
f_l&\mbox{for $l\neq i$},\\
f_if_j&\mbox{for $l=i$},
\end{cases}
\end{equation*}
lies in ${\rm Aut}(F_n)$.

The rest of the statements follow directly from the definitions and the fact that
each element of $\Sigma$ is written as
a noncommutative Laurent monomial in $s_1,\ldots, s_n$.
\end{proof}


\subsection{The existence of categorical quotient}

 \begin{proposition}\label{existe}
 Let $\Sigma$ be a finitely generated group, let $G$ be an algebraic group \textup (not necessarily connected or
  affine\textup), let $R$ be an al\-geb\-raic subgroup of ${\rm Aut}(G)$,
  and let $X={\rm Hom}(\Sigma, G)$.
The categorical quotient {\rm\eqref{Lunan}} exists in each of the following two cases:
\begin{enumerate}[\hskip 4.2mm\rm (A)]
\item[\rm (F)] $R$ is finite;
\item[\rm (R)]  $G$ is affine and $R$ is reductive.
\end{enumerate}
If {\rm (F)} holds, then the categorical quotient {\rm\eqref{Lunan}} is the geometric quotient.
 If {\rm (R)} holds, then the variety $X/\!\!/R$ is affine. In each of cases {\rm (F)} and {\rm (R)}, the morphism $\pi_{X/\!\!/R}$
is surjective.
 \end{proposition}
 \begin{proof}
 According to \cite{Ba54}, the variety $\underline{G}$ is quasi-pro\-jec\-tive. Hence, $X$, being closed in the product of several copies of $\underline{G}$,
  is quasi-projective as well. This implies the existence of the geometric factor
  in case (F) (see \cite[Chap.\,III, Sect. 12, Prop. 19, Ex.\,2]{Se97}). This factor is automatically categorical with the surjective morphism $\pi_{X/\!\!/R}$
  (see \cite[4.3]{PV94}, \cite[Sect.\,II, \S 6]{Bo91}).

 In case (R), the variety
  $\underline{G}$
 is affine. In view of the remark on closedness made in the previous paragraph,
  $X$ is affine as well. Ac\-cord\-ing to \cite[Chap.\,1, \S2]{MF82},
  from this and the reductivity of $R$ it follows the existence of the categorical quotient \eqref{Lunan}, the affineness of $X/\!\!/R$, and the surjectivity of $\pi_{X/\!\!/R}$.
\end{proof}

\subsection{The kernel of the action of \boldmath ${\rm Aut}(\Sigma)$ on ${\rm Hom}(\Sigma, G)/\!\!/R$ in cases (F) and (R): geometric description}\label{kerr}

Let $\Sigma$ be a finitely generated group and let $G$ be an algebraic group (not necessarily con\-nected or affine). Having fixed a system of $n$ gene\-ra\-tors in $\Sigma$, we identify $X={\rm Hom}(\Sigma, G)$ with a closed subset of $G^n$ as described in Section \ref{wds}. For any $w\in \Sigma$, $\gamma\in {\rm Aut}(G)$ and $i\in\{1,\ldots, n\}$, the closed set
\begin{equation}\label{worddn}
X_{w, \gamma, i}:=\{x=(g_1,\ldots, g_n)\in X\mid w(x)=\gamma(g_i)\}
\end{equation}
  is the fiber over $e$ of the morphism
$$X\to G, \quad x=(g_1,\ldots, g_n)\mapsto w(x)\gamma(g_i)^{-1}.$$
As it contains $(e,\ldots, e)$, it is nonempty.

From
\eqref{mornew} and \eqref{worddn} it follows that for any $\sigma\in {\rm Aut}(\Sigma)$ we have
\begin{equation}\label{intersee}
\textstyle \bigcap_{i=1}^n X_{\sigma(f_i), \gamma, i}=\{x\in X\mid \sigma_X(x)=\gamma(x)\}.
\end{equation}

The following Lemmas \ref{lemma} and \ref{categ} describe the kernel of the action of ${\rm Aut}(\Sigma)$ on ${\rm Hom}(\Sigma, G)/\!\!/R$ respectively in cases (F) and (R).

\begin{lemma}\label{lemma}
We retain the notation and conventions introduced in this section.
Let $R$ be a finite subgroup of ${\rm Aut}(G)$.
The following properties of an element $\sigma\in {\rm Aut}(\Sigma)$ are equivalent:
\begin{enumerate}[\hskip 4.2mm\rm(a)]
\item\label{nonef1} $\sigma$ lies in the kernel of the action of  ${\rm Aut}(\Sigma)$ on $X/\!\!/R$;
\item\label{nonef6} $\sigma_X(\mathcal O)=\mathcal O$ for every $R$-orbit $\mathcal O$ in $X$;
\item\label{nonef2}
    for every irreducible component $Y$ of the variety $X$ there is an element $\gamma\in R$ such that
\begin{equation}\label{===n}
Y\subseteq \textstyle \bigcap_{i=1}^n X_{\sigma(f_i), \gamma, i}.
\end{equation}
\end{enumerate}
\end{lemma}

\begin{proof}
In view of Proposition \ref{existe}, each fiber of the mor\-phism $\pi_{X/\!\!/R}$ is an $R$-orbit in $X$ and vice versa.
Since the actions of ${\rm Aut}(\Sigma)$ and $R$  on $X$ commute, it follows from \eqref{pssp1} that
for each point $b\in X/\!\!/R$, the restriction of the morphism $\sigma_X$ to the orbit
$\pi_{X/\!\!/R}^{-1}(b)$ is
its $R$-equivariant isomorphism with the orbit
$\pi_{X/\!\!/R}^{-1}(\sigma_{X/\!\!/R}(b))$. This proves the equivalence of the conditions \eqref{nonef1}
and \eqref{nonef6} and, given
\eqref{intersee},
their equivalence
to the equality
\begin{equation}\label{X=si}
\textstyle X=\bigcup_{\gamma\in R}\big(\bigcap_{i=1}^n X_{\sigma(f_i), \gamma, i}\big).
\end{equation}

\eqref{nonef1}$\Rightarrow$\eqref{nonef2}
If the equality \eqref{X=si} holds, then each irreducible component $Y$ of $X$ is the union of closed subsets of the form
\begin{equation}
\textstyle Y\cap \big(\bigcap_{i=1}^n X_{\sigma(f_i), \gamma, i}\big),\;\; \mbox{where $\gamma\in R$}.
\end{equation}
Since the group $R$ is finite, there are finitely many of these subsets. The irreducibility of
$Y$ therefore implies that $Y$ coincides with
one of them. Hence, \eqref{===n} holds for some $\gamma\in R$.

\eqref{nonef2}$\Rightarrow$\eqref{nonef1}
If \eqref{nonef2} holds, then the union of all irreducible components of $X$ lies on the right-hand side of the equality
\eqref{X=si}, i.e., this right-hand side contains $X$. The reverse inclusion is obvious.
Hence, the equality \eqref{X=si} holds.
\end{proof}

\begin{lemma}\label{categ}
Retaining the notation and conventions introduced in this section,
we assume that the group $G$ is affine. Let $R$ be a reductive algebraic subgroup of ${\rm Aut}(G)$.
The following properties of an element $\sigma\in {\rm Aut}(\Sigma)$ are equivalent:
\begin{enumerate}[\hskip 4.2mm\rm(a)]
\item\label{nonef3} $\sigma$ lies in the kernel of the action of ${\rm Aut}(\Sigma)$  on $X/\!\!/R$;
\item\label{nonef5} $\sigma_X(\mathcal O)=\mathcal O$ for every closed $R$-orbit $\mathcal O$ in $X$;
\item\label{nonef4} each closed $R$-orbit in $X$ belongs to the set
\begin{equation}\label{===}
\textstyle\bigcup_{\gamma\in R}\big(\bigcap_{i=1}^n X_{\sigma(f_i), \gamma, i}\big).
\end{equation}
\end{enumerate}
\end{lemma}

\begin{proof}
For every $b\in X/\!\!/R$, the fiber
$\pi^{-1}_{X/\!\!/R}(b)$
of the surjective (see Pro\-po\-sition \ref{existe})
morphism $\pi_{X/\!\!/R}$ is an $R$-invariant closed
subset of $X$, which contains a unique closed $R$-orbit $\mathcal O_b$
(see \cite[\S2 and Ap\-pend.\,1B]{MF82}. The restriction of
$\sigma_X$ to
$\pi_{X/\!\!/R}^{-1}(b)$ is an $R$-equivariant iso\-mor\-phism with
the fiber $\pi_{X/\!\!/R}^{-1}(\sigma_{X/\!\!/R}(b))$.
 In view of the uniqueness of closed orbits in the fibers,
this means that $\sigma_X(\mathcal O_b)=\mathcal O_{\sigma_{X/\!\!/S}(b)}$. Therefore, the equalities $\sigma_{X/\!\!/S}(b)=b$
and $\sigma_X(\mathcal O_b)=\mathcal O_b$ are equivalent. This proves
\eqref{nonef3}$\Leftrightarrow$\eqref{nonef5}. In turn, this and
\eqref{intersee} imply
\eqref{nonef3}$\Leftrightarrow$\eqref{nonef4}.
\end{proof}

\begin{corollary}\label{intker}
If the conditions of Lemma
{\rm\ref{categ}} hold and $R={\rm Int}(G)$, then
${\rm Int}(\Sigma)$ lies in the kernel of the action of ${\rm Aut}(\Sigma)$ on $X/\!\!/R$.
\end{corollary}
\begin{proof}
If $\sigma\in {\rm Int}(\Sigma)$, then \eqref{mornew} implies that $x$ and $\sigma_X(x)$ lie in the same $R$-orbit for each $x \in X$. The assertion therefore follows from the equivalence
of conditions \eqref{nonef3} and \eqref{nonef5} in Lemma \ref{categ}.
\end{proof}

\subsection{The faithfulness of the action of \boldmath ${\rm Aut}(F_n)$ on
${\rm Hom}(F_n, G)$: algebraic
criterion}

 \begin{theorem}\label{criterion}
 Let $G$ be a group and let $X\!=\!{\rm Hom}(F_n,G)$.
 The fol\-low\-ing properties are equivalent:
  \begin{enumerate}[\hskip 4.2mm\rm(a)]
  \item the action of
${\rm Aut}(F_n)$ on $X$ is faithful;
  \item if $n\geqslant 2$, then in an $n$-letter alphabet
  there is no nonempty irreducible word that is the identity
  in $G$, and if $n=1$, then $G$ contains an element of order $\geqslant 3$.
  \end{enumerate}
 \end{theorem}

 \begin{proof}
 We use the notation of Section \ref{wds} with $\Sigma=F_n$ and $s_j=f_j$ for all $j$.

  For $n=1$, the equivalence of (a) and (b) follows from Proposition \ref{pro}(f).
  Consider the case $n\geqslant 2$.

 (a)$\Rightarrow$(b)
 Suppose, arguing by contradiction, that (a) holds, but in an $n$-letter alphabet
there exists a nonempty irreducible word that is the identity
in $G$. So there is
  a nonidentity element $w\in F_n$ such that
  \begin{equation}\label{w(x)=2n}
  w(x)=e\;\; \mbox{for each $x\in X$}.
  \end{equation}

 The element $\sigma:={\rm int}_w\in {\rm Aut}(F_n)$ is different from the identity because the group ${\mathscr C}(F_n)$ is trivial for $n\geqslant 2$ (cf. \cite[Chap.\,I, Prop. 2.19]{LS77}), and $w\neq e$.
  However, from \eqref{w(x)=2n}
  it follows that $\sigma_X(x)=x$ for each $x\in X$, i.e.,
  that $\sigma$ lies in the kernel of the action of ${\rm Aut}(F_n)$ on $X$. This contradicts\;(a).

 (b)$\Rightarrow$(a)
Suppose, arguing by contradiction, that (b) holds, but
  the kernel of the action of ${\rm Aut}(F_n)$
on $X$  contains a nonidentity element $\sigma\in {\rm Aut}(F_n)$, so that we have (see \eqref{mornew}),
  \begin{equation}\label{ba}
  \sigma(f_i)(x)=f_i(x) \;\;\mbox{for all $x\in X$ and $i$.}
  \end{equation}
In view of $\sigma\neq e$, there exists $f_j$ for which $\sigma(f_j)\neq f_j$, i.e., $w:=
\sigma(f_j)f_j^{-1}$ is a nonidentity element of the group $F_n$. At the same time, \eqref{ba} implies that this $w$ satisfies condition \eqref{w(x)=2n}. Therefore, in
the alphabet $f_1,\ldots, f_n$ there is a
nonempty irreducible word that is the identity in $G$. This contradicts (b).
 \end{proof}

 \begin{corollary}\label{solv}
 For each virtually solvable group $G$, the action
  of ${\rm Aut}(F_n)$ on $X:={\rm Hom}(F_n, G)$,
  $n\geqslant 2$, it is nonfaithful.
 \end{corollary}
 \begin{proof}
  By the definition of virtual solvable group, $G$ has a solvable sub\-group $S$ of a finite index $d$. We can (and shall) assume that
  $S$ is normal, replacing it with the intersection of all subgroups conjugate to it. Since $S$ is solvable, in the alphabet of two letters $x, y$ there exists a
  nonempty irreducible word $r(x, y)$, which is the identity
  in $S$ (see\;\cite[14.65]{Ne67}).
   It follows from the normality of $S$ that $g^d\in S$ for each $g\in G$.
   Hence, the nonempty irreducible word $r(x^d, y^d)$ is
   the identity in $G$. The claim now follows from
  Theorem  \ref{criterion}.
 \end{proof}

\subsection{
Proof of Theorem \ref{fg1}: the case of trivial
subgroup \boldmath $R$}\label{p0}

  To prove Theorem \ref{fg1}, we first need to consider a special case of trivial subgroup $R$.
   We will prove a more general statement concerning
   not only algebraic groups, but also real Lie groups.

 \begin{theorem}\label{noquon}
Let $X={\rm Hom}(F_n, G)$, $n\geqslant 2$, and
let $G$ be either an algebraic group
\textup (not necessarily connected or affine\textup) or
  a real Lie group with a finite number of connected components.
  Then the following pro\-per\-ties are equivalent:
  \begin{enumerate}[\hskip 4.2mm \rm(a)]
\item the action of
${\rm Aut}(F_n)$ on $X$ is faithful;
\item
the group $G^0$ is nonsolvable.
\end{enumerate}
If $G$ is a real Lie group, then the implication {\rm(b)$\Rightarrow$(a)} is true even without the condition that the number of its connected components is finite.
 \end{theorem}

 \begin{proof}
If $G$ is a real Lie group, then
  \begin{equation}\label{inf}
  [G:G^0]<\infty
  \end{equation}
  by the condition. If $G$ is an algebraic group, then \eqref{inf} is satisfied auto\-matically.
  It follows from \eqref{inf} that if $G^0$ is solvable, then $G$ is virtually solvable. Together with Corollary \ref{solv}, this proves implication
(a)$\Rightarrow$(b).

 (b)$\Rightarrow$(a)
Let the group $G^0$ be nonsolvable. In view of  Theorem \ref{criterion}, it is required to prove that in an alphabet of $n$ letters there is no non\-empty irreducible word that is
   the identity relation in $G$. Arguing by contradiction, suppose that such a word exists. Hence, there is a non\-trivial element $w\in F_n$ with the property\;\eqref{w(x)=2n}.

 Let $G$ be a connected real Lie group. Then,
   due to nonsolvability, $G^0$ contains a free subgroup of rank $n$ (see \cite[Thm.]{Ep71}). Let $g_1,\ldots, g_n$ be its free system of generators.
Then $w(g_1,\ldots, g_n)=e$ due to \eqref{w(x)=2n}, which contradicts the absence
of nontrivial relations between $g_1,\ldots, g_n$.

 Let now $G$ be an algebraic group.
By Chevalley's theorem, the algeb\-raic group $G^0$ contains the largest connected affine normal subgroup $G^0_{\rm aff}$, and $G^0/G^0_{\rm aff}$ is an Abelian variety.
Since the group $G^0$ is nonsolv\-able and the group $G^0/G^0_{\rm aff}$ is commutative (and therefore solvable), the group
$G^0_{\rm aff}$ is nonsolvable. Hence, $G^0_{\rm aff}$ does not coincide with its radical ${\rm Rad}(G^0_{\rm aff})$, and therefore,
$G^0_{\rm aff}/{\rm Rad}(G^0_{\rm aff})$ is a nontrivial connected semi\-simple algebraic group. This reduces the proof to the case where  $G$ is a nontrivial connected semisimple algebraic group. We will therefore further assume that this condition is met.
In view of \cite[Thm.\,B]{Bo83}, from it and the inequality $n\geqslant 2$
it follows that the morphism
\begin{equation*}\label{boldw}
X\to {G},\quad x\mapsto w(x)
\end{equation*}
 is dominant. In view of \eqref{w(x)=2n}, this means that
 the group
 $G$ is trivial, which is a contradiction.
 \end{proof}

 \begin{remark}
If $G$ is an nonsolvable algebraic group and the field $k$ is uncountable,
then $G$ contains a free subgroup of any finite rank
(see \cite[Thm.\,1.1]{BGGT12}, \cite[App.\,D]{BGGT15}), which means
that the same proof of the implication (b)$\Rightarrow$(a) in Theorem \ref{noquon}
as in the case of a real Lie group goes through. This proof is
given in the first  version \cite{Po21} of the present paper. However, in the general case,
$G$ may not contain a free subgroup (for example, this is the case for $G={\rm SL}_d$ if $k$ is the algebraic closure of a finite field, since then the order of every element of ${\rm SL}_d$
is finite).
\end{remark}

\begin{remark}
Without the condition that the number of connected components is finite, the implication (a)$\Rightarrow$(b) in Theorem \ref{noquon} is false. Indeed, take as $G$ the group $F_n$ considered as a real Lie group with $G^0=\{e\}$. Then ${\rm id}_{F_n}\in {\rm Hom}(F_n, F_n)=X$ and, for any $\sigma \in {\rm Aut}(F_n)$, we have
$\sigma_X({\rm id}_{F_n})=\sigma$ (see \eqref{definii}). Therefore, (a) holds, but (b) does not.
\end{remark}

\subsection{Proof of Theorem \ref{fg1}: general case}\label{p1}

 In view of the surjectivity and ${\rm Aut}(F_n)$-equivariance of the morphism $\pi_{X/\!\!/R}$ (see \eqref{Lunan}),
  the impli\-cation (a)$\Rightarrow$(b) follows from Theorem\;\ref{noquon}.

(b)$\Rightarrow$(a)
Let the group $G^0$ be nonsolvable. Arguing by contradiction, suppose that a nonidentity element $\sigma\in {\rm Aut}(F_n)$
lies in the kernel of the action of ${\rm Aut}(F_n)$ on $X/\!\!/R$. The variety $X$ is isomorphic to $\underline{G}^n$. It is clear that $(\underline{G}^0)^n$ is one of the irreducible components of the variety
$\underline{G}^n$. By virtue of what was said in Section \ref{wds}, this implies that $X^0:={\rm Hom}(F_n, G^0)$ is an ${\rm Aut}(F_n)$-invariant irreducible component of the variety $X$. In turn, in view of Lemma \ref{lemma} and formulas \eqref{intersee}, \eqref{mornew}, this implies the existence of an element $\gamma\in R$ such that
for every $i\in \{1,\ldots, n\}$, the following group identity  holds in $G^0$:
\begin{equation}\label{condi1}
\sigma(f_i)(g_1,\ldots, g_n)=\gamma(g_i)
\quad \mbox{for any $g_1,\ldots, g_n\in G^0$.}
\end{equation}
In particular, for every $g\in G^0$,
the equality obtained by substituting
$g_1=\cdots=g_n=g$ in \eqref{condi1} holds. Since $\sigma(f_i)$
has the form \eqref{Lora}, this means
the existence of an integer $d$ such that
the following group identity holds:
\begin{equation}\label{sgs1}
g^d=\gamma(g)
\quad\mbox{for each $g\in G^0$.}
\end{equation}

Notice that
\begin{equation}\label{neq}
d\neq 1\quad \mbox{and}\quad d\neq -1.
\end{equation}

Indeed, if $d\!=\!1$ then
from \eqref{sgs1} and \eqref{condi1}
it follows that $\sigma_{X^0}\!=\!{\rm id}_{X^0}$, i.e., $\sigma$ lies in the kernel of the action of ${\rm Aut}(F_n)$ on $X^0$.
Since $\sigma$ is a nonidentity element, and the group
$G^0$ is nonsolvable, this contradicts Theorem\;\ref{noquon}.

If $d=-1$, then
for any $g, h\in G^0$, the
equality
\begin{equation*}
h^{-1}g^{-1}=(gh)^{-1}\overset{\eqref{sgs1}}{=}\gamma(gh)
=\gamma(g)
\gamma(h)
\overset{\eqref{sgs1}}{=}g^{-1}h^{-1}
\end{equation*}
holds, meaning that the group $G^0$ is commutative contrary to its non\-so\-lvability.

Further, for any
positive integer $m$,  we obtain from \eqref{sgs1} by induction
the following group identity:
\begin{equation}\label{idenex}
g^{d^m}=\gamma^m(g)\quad \mbox{for each $g\in G^0$.}
\end{equation}

Since the group $R$ is finite, the order of
$\gamma$
is finite. Let $m$ in \eqref{idenex} be equal to this order. Then
\eqref{idenex} becomes the group identity
\begin{equation}\label{idenex1}
g^{d^m-1}=e\quad \mbox{for every $g\in G^0$.}
\end{equation}

Since $d^m-1\neq 0$ due to \eqref{neq}, from
\eqref{idenex1} we infer that
$G^0$ is a torsion group whose element orders are bounded from above.
Let us show that this contradicts the properties of the group $G^0$.

Indeed, as in the proof of Theorem \ref{noquon}
(whose notation we retain), the affine algebraic group $G^0_{\rm aff}$ is nonsolvable. Hence,
it contains a nontrivial semisimple element, and therefore, a torus of positive dimen\-sion (see \cite[Thms. 4.4, 11.10]{Bo91}).
But the set of orders of elements of the torsion subgroup of any torus of positive dimension is not bounded (see \cite[Prop.\,8.9]{Bo91}). This gives the required contradiction.
\quad $\square$

\subsection{Proofs of Corollary \ref{cor0n} and Proposition \ref{prpr}}\label{p2}

\begin{proof}[Proof of Corollary  {\rm\ref{cor0n}}]
Statements  \eqref{aaaa} and \eqref{bbbb} follow from Theorem \ref{fg1} and the next Proposition \ref{linea}.
\end{proof}

\begin{proposition}\label{linea}
Assume that a group
$H$ contains
${\rm Aut}(F_n)$. Then
\begin{enumerate}[\hskip 4.2mm\rm(i)]
\item\label{h1} $H$ contains $F_2$
  if $n\geqslant 2$;
\item\label{h2} $H$ contains $B_n$ if $n\geqslant 3$;
\item\label{h3} $H$ is not amenable if
$n\geqslant 2$;
\item\label{h4} $H$ is nonlinear if $n\geqslant 3$.
\end{enumerate}
\end{proposition}

\begin{proof}\

If $n\geqslant 2$, then  ${\mathscr C}(F_n)$
is trivial and therefore,
${\rm Int}(F_n)$
is isomorphic to $F_n$.
This gives \eqref{h1}.

If $n\!\geqslant\!  3$, then ${\rm Aut}(F_n)$ contains $B_n$ (see \cite[Chap.\,3, 3.7]{MKS66}) and is nonlinear (see \cite{FP92}). This gives \eqref{h2} and \eqref{h4}.

\eqref{h1} implies \eqref{h3}.
\end{proof}

\begin{proof}[Proof of Proposition  {\rm\ref{prpr}}]
If $G$ is affine, then the rationality of $X=G^n$
follows from the rationality of
$\underline{G}$ (see \cite[Cor. 14.14]{Bo91}).

Let $G$ be nonaffine. Arguing by contradiction, suppose that $X=G^n$
is unirational. Let us use the notation of the proof of Theorem \ref{noquon}.
The variety ${{G/G_{\rm aff}}}$ is unirational
in view of the surjectivity of the composition of the following morphisms
\begin{equation*}
X=G^n\xrightarrow{\alpha} G\xrightarrow{\beta} G/G_{\rm aff},
\end{equation*}
where $\alpha$ is a projection onto some
factor,
and $\beta$ is the canonical pro\-jection.
By the condition, $G_{\rm aff}\neq G$, so that $G/G_{\rm aff}$ is a nontrivial Abelian variety. Since such varieties
are nonunirational (see \cite[Chap.\,3, Sect. 6.2, 6.4]{Sh07}),
we get a contradiction.
\end{proof}

\subsection{Proof of Theorem \ref{stratio1}
}\label{p3}

We use in the proof of  Theorem \ref{stratio1} the following known statement (see, e.g., \cite[Thm.\,1]{Po13}).

\begin{lemma}\label{strat}
If the field of invariant rational functions of some faith\-ful linear action of a finite group on a finite-dimensional vector space over $k$ is stably rational over $k$, then the same property holds for any other such action of this group.
\end{lemma}

\begin{proof}[Proof of Theorem {\rm \ref{stratio1}}]
Consider one of the  pairs $(K, \iota)$ found in \cite{Sa84}, where $K$ is a finite group, and
\begin{equation*}
\iota\colon K\hookrightarrow {\rm GL}(V),
\end{equation*}
is a group embedding, where $V$ is a finite-dimensional vector space over $k$,
for which the field of $\iota(K)$-invariant rational functions on $V$ is not stably rational
over $k$.

In view of Lemma \ref{strat}, replacing $V$ and $\iota$ if necessary, we can
(and shall) assume that
\begin{equation}\label{ncent}
\iota(K)\cap {\mathscr C}\big({\rm GL}(V)\big)=\{{\rm id}_V\}.
\end{equation}
Indeed, let
$L$ be a one-dimensional vector space over $k$. Since we have
${\mathscr C}\big({\rm GL}(V\oplus L)\big)=\{c\cdot {\rm id}_{V\oplus L}\mid c\in k, c\neq 0\},$ the group embedding
\begin{equation*}
\iota'\colon K\hookrightarrow {\rm GL}(V\oplus L),\quad f\mapsto \iota(f)\oplus{\rm id}_L,
\end{equation*}
has the property $\iota'(K)\cap {\mathscr C}\big({\rm GL}(V\oplus L)\big)=\{{\rm id}_{V\oplus L}\} $.

It follows from \eqref{ncent} that the diagonal linear action of
$\iota(K)$ on the vector space ${\rm End}(V)^{\oplus n}$ by conjugation is faithful.
Therefore, in view of Lemma \ref{strat}, the field of
$\iota(K)$-invariant rational functions on
${\rm End}(V)^{\oplus n}$
is not stably rational over $k$.
But ${\rm GL}(V)^{n}$ is a
$\iota(K)$-invariant open subset of ${\rm End}(V)^{\oplus n}$. It is $\iota(K)$-equivariantly isomorphic to the algebraic variety
${\rm Hom}(F_n, {\rm GL}(V))$. Hence the field of $R$-invariant rational func\-tions on $X$ is not stably rational over $k$. But this field is isomorphic to the field of rational functions on $X/\!\!/R$, because, by Proposition \ref{existe}, the categorical factor \eqref{Lunan} is geometric.
Hence the variety $X/\!\!/R$ is not
stably rational. It is affine due to the affinness of ${\rm GL}(V)$
(see Proposition \ref{existe}).
Finally, by Theorem \ref{fg1}, from the nonsolvability of ${\rm GL}(V)$ with $n\geqslant 2$
it follows that the action of ${\rm Aut}(F_n)$ on
$X/\!\!/R$ is faithful.
\end{proof}

\begin{remark}\label{stra}
Found in \cite{Sa84}, the first examples of groups,
which can be taken as $K$ in Theorem \ref{stratio1}, have order $p^9$. At present, all groups of order $p^5$ with the specified property have been found
(See details and references in \cite[p.\,414, Rem.]{Po13}). For example, for $p\geqslant 5$, one of them is the group $K=\langle g_1, g_2, g_3, g_4, g_5\rangle$ of order $p^5$ given by the following con\-di\-tions
(in which $[a,b]:=a^{-1}b^{-1}ab$):
\begin{gather*}
{\mathscr C}(F)\!=\!\langle g_5\rangle,\;
g_i^p\!=\!e\;\mbox{for each $i$},\\
[g_1, g_1]\!=\!g_3,\;[g_3, g_1]\!=\!g_4,\;[g_4, g_1]\!=\![g_3, g_2]\!=\!g_5,\;[g_4, g_2]\!=\![g_4, g_3]\!=\!e.
\end{gather*}
\end{remark}

\subsection{Proofs of Theorem \ref{Cr1nn} and Corollary \ref{boundsn}}\label{p4}

\begin{proof}[Proof of Theorem {\rm\ref{Cr1nn}}]
In view of the finiteness of $R$, it follows from Proposition \ref{existe} that \eqref{Lunan} is the geometric quotient, and the variety $X/\!\!/R$ is affine (and irreducible due to
the connectedness of $G$). In particular, the fibers of the
surjective morphism \eqref{Lunan} are $R$-orbits and therefore zero-dimensional.
This implies the claim about the dimension, since
$\dim (G)=3$ and
$X=G^n$.
It remains to prove the rationality.

In this case, ${\rm Aut}(G)={\rm Int}(G)$.
Consider the adjoint action of ${\rm Int}(G)$ on $\mathfrak g$ and the diagonal
actions of the group ${\rm Int}(G)$ on
$X=G^n$ and ${\mathfrak g}^{n}:=
{\mathfrak g}\oplus\cdots\oplus {\mathfrak g}$ ($n$ summands).
According to \cite[Thm.\,1.28]{LPR06}, ${\rm SL}_2$ and ${\rm PSL}_2$ are the Cayley groups. Therefore, there exists an ${\rm Int}(G)$-equivariant (hence,
$R$-equivariant) birational mapping
\begin{equation}\label{Cayley}
X\dashrightarrow {\mathfrak g}^{n}.
\end{equation}
Consequently, the fields of $R$-invariant rational functions on
$X$ and $\mathfrak g^n$ are isomorphic. Hence the geometric quotients $X/\!\!/R$
and ${\mathfrak g}^{n}/\!\!/R$ are birationally isomorphic. But the linearity of the action of $R$ on ${\mathfrak g}^{n}$ and the decomposition ${\mathfrak g}^{n}={\mathfrak g}\oplus{\mathfrak g}^{n-1} $ with the $R$-invariant summads imply, in view of the
No-name lemma (see \cite[Lem.\,1]{Po13}, \cite[Thm.\,2.13]{PV94}),
that ${\mathfrak g}^{n}/\!\!/R$ is birationally isomorphic to
$\big({\mathfrak g}/\!\!/R\big)\times {\mathbb A}^{(n-1)\dim(\mathfrak{g})}$.
Since $\dim(\mathfrak{g})=3$ implies
the rationality of
${\mathfrak g}/\!\!/R$
(see \cite[Thm.\,2]{Mi71}), this shows
that ${\mathfrak g}^{n}/\!\!/R$, and therefore, also $X/\!\!/R$, is rational.
\end{proof}

\begin{proof}[Proof of Corollary {\rm\ref{boundsn}}]
This claim follows from Definition \ref{Defff},
Theo\-rem \ref{Cr1nn}, and Corollary \ref{cor0n}.
\end{proof}

\subsection{Proof of Theorem
\boldmath \ref{S=Gn}}\label{p5}

For $n\geqslant 2$, the group
${\rm Int}(F_n)$ is non\-tri\-vial, and hence, by Corollary \ref{intker}, the action
of ${\rm Aut}(F_n)$ on $X/\!\!/R$ is nonfaithful.

Now, let $n=1$, so that $X=G$, and ${\rm Aut}(F_n)$ is the group of order two. Let $\sigma\in {\rm Aut}(F_n)$, $\sigma\neq e$. Then $\sigma(f_1)=f_1^{-1}$, so $\sigma_X(g)=g^{-1}$ for any $g\in X$. Each fiber of the morphism
\eqref{Lunan} contains a single orbit consisting of semisimple elements, and it is the only closed orbit in this fiber (see \cite{St65}).
Since $R={\rm Int}(G)$, from here and Lemma\;\ref{categ}
the equivalence of the following properties follows:
\begin{enumerate}[\hskip 2.0mm\rm(i)]
\item $\sigma$ lies in the kernel of the  action of ${\rm Aut}(F_n)$ on $X/\!\!/R$;
\item $g$ and $g^{-1}$ are conjugate for each semisimple element $g\in G$.
\end{enumerate}

Since the intersection of any semisimple
conjugacy class
with a fixed maximal torus $T$ of $G$ is nonempty (see \cite[Thm. 11.10]{Bo91}) and
is the orbit of the normalizer of this torus (see \cite[6.1]{St65}),
property (ii) is equivalent to the fact that the Weyl group $W$
of the group $G$ considered as a subgroup of the group ${\rm GL}(\mathfrak{t})$,
contains $-1$. This, in turn, is equivalent to the fact that $-1$ is contained in the Weil group of every nontrivial connected simple normal subgroup of the group $G$.
Let $C$ be a Weyl chamber in $\mathfrak{t}$.
Since $-C$ is also a Weyl chamber, the simple transitivity of the action of  $W$ on the set of all Weyl chambers implies the existence of a unique element $w_0\in W$ such that $w_0(C)=-C$ . In view of $(-1)(C)=-C$, this means that the inclusion of $-1\in W$ is equivalent to the equality
 $w_0=-1$.
In \cite[Table. I--IX]{Bo68}, the explicit description of the element $w_0$ is given for every connected simple algebraic group. It follows from it that the equality $w_0= -1$ for such a group is equivalent to the fact that
the type of this group is not contained in list \eqref{type}.
This completes the proof.
\quad $\square$

\subsection{Final remarks}\

(a) Corollary \ref{boundsn} concerns, in particular, the subgroups of the Cre\-mona groups. Taking this opportunity, we will supplement it here with a remark on S. Cantat's question
about these subgroups.

In \cite{Co17}, are given examples of finitely generated (and even finitely presented) groups nonembeddable into any Cremona group, which ans\-wers S. Cantat's question about the existence of such groups (see also \cite{Ca131} ). These examples are based on the fact that the word problem (\cite[Thm.\,1.2]{Co13}) is solvable in every finitely generated sub\-group of any Cremona group. Let us indicate another way to answer this question
(and even in a stronger form, with the addition of the group simplicity condition).

Namely, we recall \cite[Def.\,1]{Po14} that a group $H$ is called {\it Jordan} if there exists a finite set $\mathcal F$ of finite groups such that every finite subgroup of $H $ is an extension of an Abelian group by a group taken from $\mathcal F$.
According to \cite[p.\,188, Exmp.\,6]{Po14}, the R.\;Thompson group $V$ is an example of a non-Jordan finitely presented group. Since any Cremona group is Jordan (see\,\cite[Cor.\,1.5]{Bi16}, \cite{PS16}),
the group $V$ is nonembeddable into it. Furthermore,
in addition to this property, $V$ is simple, and therefore, every homomorphism of  $V$ into any Cremona group is trivial (unlike \cite{Co13}, this proves
\cite[Cor.\,1.4]{Co13} without using the obtained in \cite{Mi81} amplification
of the Boone--Novikov const\-ruction).

We note that, after \cite{Co13},  many unrelated to the word problem examples of finitely generated (and even finitely presented) groups nonembeddable into any Cremona group were obtained in \cite{CX18}. However,
basing on the currently available (July 2022) information,
it is impossible to deduce from \cite[Thms.\,C and 7.15]{CX18} that
$V$ is nonembeddable into any Cremona group.
Indeed, the group $V$ does not have Kazhdan's property (T) (see\;\cite{BJ19}),
and whether it has property $(\tau^\infty)$ (see\,\cite[Sect.\,7.1.3]{CX18})
is unknown \cite{Co22}.

\vskip 2mm

(b) Among the irreducible affine varieties $X/\!\!/R$ whose automorphism group contains ${\rm Aut}(F_n)$, there are open subsets of affine spaces. Indeed, by Theorem \ref{fg1} for $n\geqslant 2$, such an example is $X/\!\!/R$ with $G={\rm GL}_d$, $d\geqslant 2$, and trivial $R$. The following construction generalizes this example.

Consider a finite dimensional associative $k$-algebra $A$ with identity. The group $A^*$ of its invertible elements is a connected affine algebraic group whose underlying variety is open in $A$. If
$A^*$ is nonsolvable, then in view of  Theorem \ref{fg1},
it can be taken instead of ${\rm GL}_d$ in the example from the previous paragraph.

\vskip 1mm

(c)
In \cite[pp.\,272]{CX18}, is obtained the lower bound
\begin{equation*}\label{etemat}
   n-2\leqslant {\rm Var}_{\mathbb C}\big({\rm Out}(F_n)\big).
\end{equation*}
The following theorem yields, among other things,
an upper bound.

\begin{theorem}\label{Out}
We retain the notation of Theorem {\rm \ref{fg1}}.\;Let ${\rm char}(k) =0$, $n\geqslant 3$, $G={\rm SL}_2$ or ${\rm PSL}_2$, and $R={\rm Int}(G)$. Then $X/\!\!/R$ is an ir\-re\-duc\-ible rational affine $(3n-3)$-dimensional manifold,
whose automorphism group contains ${\rm Out}(F_n)$.
\end{theorem}

\begin{proof}
The affineness of $X/\!\!/R$ follows from the reductivity of $R$.
Ac\-cord\-ing to \cite{Ho75}, for $n\geqslant 3$, the
kernel of the action of
${\rm Aut}(F_n)$ on $X/\!\!/R$ is
${\rm Int}(F_n)$, so ${\rm Out}(F_n)$ is embedded in
${\rm Aut}(X/\!\!/R)$. From \cite[Lem.\,3.3, Thm.\,4.1]{Ri88} and $\dim(G)=3$ we infer the nonemptiness of the open subsets of
$X=G^n$ and $\mathfrak g^n$ comprised by points whose  $R$-orbits are three-dimensional and closed. This and the existence of the geometric quotients for the suitable open subsets of
 $X=G^n$ and $\mathfrak g^n$ (see \cite[Thm.\,4.4]{PV94}) imply
that the dimensions of the varieties $G^n/\!\!/R$ and $\mathfrak g^n/\!\!/R$ are equal to $3n-3$, and their fields of rational functions
coincide under the natural embedding
with the fields of $R$-invariant rational functions on $X$ and $\mathfrak g^n$ respectively.
As in the proof of  Theorem \ref{Cr1nn}, there is an $R$-equivariant
birational mapping \eqref{Cayley}, and therefore, the specified fields of $R$-invariant rational functions are isomorphic. Hence the algebraic varieties $G^n/\!\!/R$ and $\mathfrak g^n/\!\!/R$ are birationally isomorphic.
But, according to P.\,Katsylo, the field of inva\-riant rational functions on any finite dimensional algebraic ${\rm SL}(2)$-mo\-dule is purely transcendental over $k$ (see \cite[Thm.\,2.12]{PV94}). Hence $G^n/\!\!/R=X/\!\!/R$ is a rational algebraic variety.
\end{proof}

\begin{corollary} If ${\rm char}(k)=0$ and $n\geqslant 3$, then
\begin{equation}\label{estim}
\mbox{${\rm Var}_k\big({\rm Out}(F_n)\big)\leqslant 3n-3$\quad and\quad
${\rm Crem}_k\big({\rm Out}(F_n)\big)\leqslant 3n-3$.}
\end{equation}
\end{corollary}
As noted in \cite[pp.\,272]{CX18} (with reference to \cite{MS75}), over $\mathbb C$, the minimal dimension in which ${\rm Out}(F_n)$
is the group of birational self-maps,
does not exceed $6n$.
The right-hand side inequality in \eqref{estim} is the twice stronger upper bound.


\begin{thebibliography}{BGGT12}

\bibitem[Ba54]{Ba54} I. Barsotti, {\it A note on abelian varieties}, Rend. Cir. Palermo {\bf 2} (1954), 1--22.

\bibitem[BL83]{BL83}  H.\;Bass, A.\;Lubotzky, {\it Automorphisms of groups and of schemes of finite type}, Israel J. Math. {\bf 44} (1983), no. 1, 1--22.



\bibitem[Bi16]{Bi16} C. Birkar,
{\it Singularities of linear systems and boundedness of Fano va\-ri\-eties},
{\tt {\scriptsize arXiv}:1609.05543} (2016).

\bibitem[Bo83]{Bo83} A. Borel, {\it On free subgroups of semi-simple groups}, Enseign. Math. (2), {\bf 29}
(1983), no. 1-2, 151--164.

\bibitem[Bo91]{Bo91}
A. Borel, {\it Linear Algebraic Groups}, Graduate Texts in Mathematics, Vol. 126, Sprin\-ger-Verlag, New York, 1991.

\bibitem[Bo68]{Bo68} N.\;Bourbaki, {\it Groupes et Alg\`ebres de Lie}, Chaps. IV, V, VI, Hermann, Paris, 1968.

\bibitem[BGGT12]{BGGT12} E. Breuillard, B. Green, R. Guralnick, T. Tao, {\it Strongly dense free sub\-groups of semi\-simple algebraic groups}, Israel J. Math. {\bf  192}
    (2012), 347--379.

\bibitem[BGGT15]{BGGT15}    E. Breuillard, B. Green, R. Guralnick, T. Tao,
{\it Expansion in finite simple groups of Lie type},
J. Eur. Math. Soc. {\bf 17} (2015), 1367--1434.

\bibitem[BJ19]{BJ19} A.\;Brothier, V.\;F.\;R.\;Jones,\;{\it On the Haagerup and Kazhdan properties of R.\;Thomp\-son's groups}, J. Group Theory {\bf 22} (2019), 795--807.

    \bibitem[Ca12]{Ca12}  S. Cantat, {\it Linear groups in Cremona groups} (2012),
{\tt {\scriptsize http:$\!/\!\!/$www.mathnet.\break ru/php/presentation.phtml?presentid=4848$\&$option$\underline{\ }$lang=eng}}.

\bibitem[Ca$13_1$]{Ca131} S. Cantat, {\it A remark on groups of birational transformations and\;the word problem \textup(af\-ter Yves de Cornulier\textup)} (2013),
 {\tt {\scriptsize   https:$/\!\!/$perso.univ-ren\break nes1.fr/serge.cantat/Articles/wpic.pdf}}.

\bibitem[Ca$13_2$]{Ca132} S. Cantat, {\it The Cremona group in two variables}, European Congress of Math.,
Eur. Math. Soc., Z \"urich, 2013, pp. 211--225.


     \bibitem[CX18]{CX18} S. Cantat, J. Xie, {\it Algebraic actions of discrete groups: the p-adic me\-thod},
    Acta Math. {\bf 220} (2018), no.\,2, 239--295.

    \bibitem[CD13]{CD13} D. Cerveau, J. D\'eserti, {\it Transformations Birationnelles de Petit Degr\'e}, Cours Sp\'ecialis\'es, Vol. 19, Collection SMF, 2013.

\bibitem[Co17]{Co17} Y. Cornulier, {\it Nonlinearity of some subgroups of the planar Cremona group}, {\tt \scriptsize arXiv:1701.00275v1} (2017).

\bibitem[Co13]{Co13} Y. Cornulier, {\it Sofic profile and computability
of Cremona groups}, Mi\-chi\-gan Math. J. {\bf 62} (2013), 823--841.

\bibitem[Co22]{Co22} Y. Cornulier, {\it Letter to V.\;L.\;Popov}, May 13, 2022.

\bibitem[DF04]{DF04} V.\;Drensky, E.\;Formanek, {\it Polynomial Identity Rings},
Springer Basel, 2004.


\bibitem[Ep71]{Ep71} D. B. A. Epstein, {\it Almost all subgroups of a Lie group are free}, J. Algebra
{\bf 19} (1971), 261--262.



 \bibitem[FP92]{FP92} E. Formanek, C. Procesi, {\it The automorphism group of a free group is not linear}, J. Algebra {\bf 149} (1992), 494--499.

     \bibitem[Go97]{Go97} W.\;M.\;Goldman, {\it Ergodic theory on moduli spaces}, Ann. of Math. {\bf 146} (1997), 475--507.

       \bibitem[Go06]{Go06}   W.\;M.\;Goldman, {\it Mapping class group dynamics on surface group re\-pre\-sentations}, in {\it Problems on Mapping Class Groups and Related Topics}, Proc. Sympos. Pure Math., Vol. 74, Amer. Math. Soc., Pro\-vi\-dence, RI, 2006, pp. 189--214.

     \bibitem[Go09]{Go09} W.\;M.\;Goldman, {\it Trace coordinates on Fricke spaces of some simple hyperbolic surfaces}, in:
         {\it Handbook of Teichm\"uller Theory}, Vol. II, IRMA Lect. Math. Theor. Phys., Vol. 13, Eur. Math. Soc., Z\"urich, 2009, pp. 611--684.

          \bibitem[Ho72]{Ho72} R. Horowitz, {\it Characters of free groups represented in the
two-di\-men\-sional special linear group}, Comm. Pure Appl. Math. {\bf XXV} (1972), 635--649.


      \bibitem[Ho75]{Ho75} R. Horowitz, {\it Induced automorphisms of Fricke characters of free groups}, Trans. Amer. Math. Soc. {\bf 208} (1975), 41--50.



      \bibitem[Hu75]{Hu75} J. E. Humphreys, {\it Linear Algebraic Groups}, Springer-Verlag, New York, 1975.

           \bibitem[Ki94]{Ki94} E. Kirchberg, {\it Discrete groups with Kazhdan's property T
and fac\-to\-rization property are residually finite}, Math. Ann. {\bf 299} (1994), 551--563.

    \bibitem[Kr02]{Kr02} D. Krammer, {\it Braid groups are linear}, Annals of Math. {\bf 155} (2002), 131--156.

        \bibitem[LPR06]{LPR06} N. Lemire, V. L. Popov, Z. Reichstein, {\it Cayley groups},
        J. Amer. Math. Soc. {\bf 19} (2006), no. 4, 921--967.

     \bibitem[LS77]{LS77} R.\;C.\;Lyndon, P.\;E.\;Schupp, {\it Combinatorial Group Theory},
     Ergebnisse der Mathematik und ihrer Grenzgebiete, Bd. 89, Springer-Verlag, Berlin, 1977.

     \bibitem[MS75]{MS75} A.\;M.\;Macbeath, D.\;Singerman, {\it Spaces of subgroups and Teichm\"uller space}, Proc. London Math. Soc. {\bf 31} (1975), 21--256.

     \bibitem[Ma80]{Ma80} W.\;Magnus, {\it Rings of Fricke characters and automorphism groups
of free groups}, Math. Z. {\bf 170} (1980), 91--103.

 \bibitem[Ma81]{Ma81} W.\;Magnus, {\it The uses of $2$ by $2$ matrices in combinatorial group theory.
A survey}, Resultate der Math. {\bf 4} (1981), 171--192.

 \bibitem[MKS66]{MKS66} W. Magnus, A. Karrass, D. Solitar, {\it Combinatorial Group Theory}, Inter\-science, New York, 1966.



  \bibitem[Mi81]{Mi81}   C. F. Miller III, {\it The word problem in quotients of a group}, in; {\it Aspects of Ef\-fec\-tive
Algebra} (Clayton, 1979), Upside Down,Yarra Glen, Victoria, 1981, pp. 246–250.

\bibitem[Mi71]{Mi71} T. Miyata, {\it Invariants of certain groups {\rm I}}, Nagoya Math. J. {\bf 41} (1971), 69--73.

  \bibitem[MF82]{MF82} D. Mumford, J. Fogarty, {\it Geometric Invariant Theory}, Second Edition,
  Ergebnisse der Mathematik und ihrer Grenzgebiete, Bd. 34, Springer-Verlag, Berlin, 1982.

   \bibitem[Ne67]{Ne67} H. Neumann, {\it Varieties of Groups}, Ergebnisse der Mathematik und ihrer Grenzgebiete, Bd. 37, Springer-Verlag, Berlin, 1967.


\bibitem[Po94]{Po94} V.\,L.\,Popov,\;{\it Sections in invariant theory},
Proc. Sophus Lie Memorial Conf. (Oslo, 1992), Scandinavian University Press, Oslo, 1994, 315--361.

 \bibitem[Po13]{Po13}     V.\;L.\;Popov, {\it Rationality and the FML invariant}, J. Ramanujan Math. Soc. {\bf 28A} (2013), 409--415.

     \bibitem[Po14]{Po14}  V.\;L.\;Popov,\;{\it Jordan groups and auto\-morphism groups of algebraic varieties}, in:    {\it Auto\-morphisms
in Birational and Affine Geometry} (Levico Terme, 2012), Springer Proc.
Math. Stat., Vol. 79, Springer, Cham, 2014, pp.\;185--213.

\bibitem[Po21]{Po21}  V.\;L.\;Popov,\;{\it Embeddings of groups ${\rm Aut}(F_n)$ into automorphism groups of algebraic varieties}, {\tt \scriptsize arXiv:2106.02072v1} (2021).

    \bibitem[Po22]{Po22}  V.\;L.\;Popov,\;{\it Underlying varieties and group structures}, Izvestiya Math., to appear (2022), {\tt \scriptsize arXiv:2105.12861v2}.

\bibitem[Po23]{Po23}   V.\;L.\;Popov,\;{\it Embeddings of automorphism groups of free groups
into auto\-morphism groups of affine algebraic varieties}, Proc. Steklov Inst. Math., to appear (2023).

\bibitem[PV94]{PV94}     V.~L.~Popov, E.~B.~Vinberg, {\it
Invariant Theory}, in: {\it Algebraic
Geometry} IV, Encycl. Math. Sci., Vol.
55, Springer Verlag, Berlin, 1994, pp. 123--284.

\bibitem[PS16]{PS16}    Yu. Prokhorov, C. Shramov, {\it Jordan property for Cremona groups},
    Ame\-rican J. Math. {\bf 138} (2016), 403--418.

    \bibitem[Ri88]{Ri88} R.\;W.\;Richardson, {\it Conjugacy classes on $n$-tuples in Lie algebras and algebraic groups}, Duke Math. J. {\bf 57} (1988), no. 1, 1--35.

\bibitem[Sa84]{Sa84} D.\;J.\;Saltman, {\it Noether's problem over an algebraically closed field}, Invent. Math. {\bf 77} (1984), 71--84.

         \bibitem[Se55]{Se55}  J.-P.\;Serre,\;{\it Faisceaux alg\'ebriques coh\'erents}, Ann. of Math. {\bf 61} (1955), 197--278.

         \bibitem[Se97]{Se97} J.-P.\;Serre, {\it Algebraic Groups and Class Fields}, Graduate Texts in Ma\-the\-ma\-tics, Vol. 117, Springer, New York, 1997.


               \bibitem[Sh07]{Sh07} I. R. Shafarevich, {\it Basic Algebraic Geometry} 1, Springer, Heidelberg, 2007.

          \bibitem[St65]{St65} R.\;Steinberg, {\it Regular elements in semisimple algebraic groups}, Publ. Math. l'IHES {\bf 25} (1965), 49--80.
\end{thebibliography}
\end{document}